\documentclass{amsart}
\usepackage{amsthm}
\usepackage{enumerate}

\theoremstyle{definition}
\newtheorem{theorem}{Theorem}

\newtheorem{corollary}{Corollary}

\newtheorem{remarks}{Remarks}

\newtheorem{lemma}{Lemma}

\usepackage{tikz}
 \usetikzlibrary{calc} 
\usepackage{float}

\begin{document}
\author{Gao Mou}
\address{School of Physical and Mathematical Sciences, Nanyang Technological University, Singapore} 
\email{gaom0002@e.ntu.edu.sg}

\title{Hamiltonian prisms on 5-chordal graphs}

\maketitle

\begin{abstract}
In this paper, we provide a method to find a Hamiltonian cycle in the prism of a 5-chordal graph, which is $(1+\epsilon)$-tough, with some special conditions.\end{abstract}

\section{Preliminaries}
The toughness of graphs is a concept introduced by Chv{\'a}tal \cite{chvatal1973tough}, when he was doing research on Hamiltonicity of graphs.
A graph $G$ is called {\em$\beta$-tough}, if for any $p\ge2$, it cannot be split into $p$ components by deleting less than $n\beta$ vertices.
It is not difficult to prove that every $k$-tough graph is $2k$-connected.
A graph $G$ is called {\em$k$-chordal}, if for any circle $C$ in $G$ with length $|C|\ge k$, $C$ has a chord in $G$. Usually, we call a 3-chordal graph a chordal graph for convenience. Clearly, a chordal graph is also $k$-chordal for any $k\ge3$. 

Chv{\'a}tal posed a famous conjecture, which is still open today, saying that there exists a constant $\beta$ such that every $\beta$-tough graph is Hamiltonian. Clearly, being 1-tough is a necessary condition for being Hamiltonian. What is more, there exist 2-tough graphs which are not Hamiltonian \cite{bauer2000not}. In recent decades, researchers found many special classes of graphs for which Chv{\'a}tal's conjecture is true, for example, chordal graphs \cite{Kabela201510}, $2K_2$-free graphs \cite{broersma2014toughness}, and planar graphs \cite{Tutte1956A}.
Moreover, researchers are also interested in many kinds of analogies of Hamiltonian cycles, such as $k$-walks, Hamiltonian-prisms, and 2-factors.
A {\em$k$-walk} in a graph $G$ is a closed walk visiting each vertex of $G$ at least once but at most $k$ times. Clearly, a Hamiltonian cycle can be considered as a 1-walk. And a $p$-walk is trivially a $q$-walk, for integers $p\le q$.
The {\em prism} over a graph $G$ is the Cartesian product $G\times K_2$ of $G$ with the complete graph $K_2$.
If $G\times K_2$ is Hamiltonian, then we say $G$ is {\em prism-Hamiltonian}, and we call $G\times K_2$ the {\em Hamilton-prism} of $G$.
It is not difficult to prove that being prism-Hamiltonian is a property stronger than admitting 2-walk but weaker than being Hamiltonian \cite{kaiser2007hamilton}.
For $k$-walks in graphs, there is also a well-known open conjecture, which is posed by Jackson and Wormald, saying that every $\frac{1}{k-1}$-tough graph admits a $k$-walk \cite{jackson1990k}.
A {\em2-factor} in a graph is a spanning subgraph, which is consisted of several disjoint cycles. An {\em edge-dominating cycle} $C$ in a graph $G$ is a cycle such that the induced subgraph on $V(G)-V(C)$ contains no edge. 

Here we list several well-known results relative chordal graphs and 5-chordals on the topic of Hamiltonicity.

\begin{theorem}\cite{Kabela201510}
Every 10-tough chordal graph is Hamiltonian.
\end{theorem}

\begin{theorem}\cite{bauer2000not}
There exists a $(\frac{7}{4}-\epsilon)$-tough chordal non-Hamiltonian graphs, for any $\epsilon>0$.
\end{theorem}

\begin{theorem}\cite{bohme1999more}
Every $(1+\epsilon)$-tough chordal planar graph is Hamiltonian, for any $\epsilon>0$. 
\end{theorem}

\begin{theorem}\cite{teska20092}
Every $(\frac{3}{4}+\epsilon)$-tough chordal planar graph admits a 2-walk, for any $\epsilon>0$.
\end{theorem}

\begin{theorem}\cite{bauer2000chordality}
Every $\frac{3}{2}$-tough 5-chordal graph contains a 2-factor.
\end{theorem}

\section{Main Results}
The main result in this paper is following:
\begin{theorem}\label{mainthms}
Let $G$ be a $(1+\epsilon)$-tough 5-chordal graph, if $G$ contains an edge-dominating cycle, then $G$ is prism-Hamiltonian.
\end{theorem}

The proof of this theorem is consisted of the following two lemmas, one of which is from \cite{gao2015on}, while another is new.
\begin{lemma}\cite{gao2015on}
Let $G$ be $(1+\epsilon)$-tough, for some $\epsilon>0$.
\begin{enumerate}
\item If $G$ contains an edge-dominating cycle $C$ with even number of vertices, then the prism over $G$ is Hamiltonian.
\item If $G$ contains an edge-dominating cycle $C=v_1v_2\cdots v_{2p+1}v_1$ of odd length, and there are three vertices $v_1$, $v_{2q}$ and $v_{2q+1}$, for some $1\le q\le p$, inducing a triangle in $G$, then the prism over $G$ is Hamiltonian.
\end{enumerate}
\end{lemma}

\begin{lemma}
Assuing $C=v_1v_2\cdots v_pv_1$ is a cycle of odd length in a 5-chordal $G$, then there exist three vertices $v_i,v_{i+2q-1},v_{i+2q}$ of $C$ (the index is in module $p$), which induce a triangle in $G$.
\end{lemma}

\begin{proof}
If $p=3$, then let $i=1$ and $q=1$, so the result holds trivially.
If $p\ge5$, then by 5-chordality, there must be an chord on $C$. Without loss of generality, we can assume that $v_1$ is one endpoint of a chord $e$.
If another endpoint of $e$ is $v_3$, then we have proved the result.
Otherwise, $e=v_1v_t$ divides $C$ into two cycles: $C_1=v_1v_2\cdots v_tv_1$ and $C_2=v_1v_t\cdots v_{t+1}v_pv_1$.

If $t$ is odd, then by mathematical induction, we can find triangle with the property we want in $C_1$.

If $t$ is even, then we can relabel the vertices of $C_2$ by $v'_1=v_1$, $v'_2=v_t$, $v'_3=v_{t+1}$ and so on. Clearly, this relabelling process does not change the odd-even of the index. So, by mathematical induction, there is a triangle we want in $C_2$, and this triangle is also the one we want in $C$.

\end{proof}
Combining the two lemmas above, we have finished the proof of Theorem \ref{mainthms}

\section{Applications}
Now, we are interested in the question that under what condition, we can find an edge-dominating cycle in a 5-chordal graph?
Here we list several well-known results on this topic.
\begin{lemma}\cite{bondy1980longest}
Let $G$ be a 2-connected graph of order $n$. If $$\delta_3(G):=\min\{\sum_{i\le3}d(x_i)|x_1,x_2,x_3 \text{are independent vertices in }G\}\ge n+2,$$ then all longest cycles in $G$ are edge-dominating.
\end{lemma}

Two edges are called {\em remote} if they are disjoint and there is no edge joining them. Let $N_G(v)$ stand for the set of vertices in $G$ which are adjacent to $v$. For an edge $e=uv$ in $G$, we define the {\em degree} $d(e)$ of $e$ by $d(e)=|N_G(u)\cup N_G(v)-\{u,v\}|$.
\begin{lemma}\cite{veldman83}
Let $G$ be a $k$-connected graph ($k\ge2$) such that, for every $k+1$ mutually remote edges $e_0,e_1,\ldots,e_k$ of $G$,
$$\sum^k_{i=0}d(e_i)>\frac{1}{2}k(|V(G)|-k).$$
Then $G$ contains an edge-dominating cycle.
\end{lemma}

\begin{lemma}\cite{Yoshimoto2008Edge}
Let $G$ be a 2-connected graph. If $d(e_1)+d(e_2)>|V(G)|-4$ for any remote edges $e_1,e_2$, then all longest cycles in $G$ are edge-dominating cycles.
\end{lemma}

Combining the three lemmas above and Theorem \ref{mainthms}, we get the following corollaries.
\begin{corollary}
Let $G$ be a $(1+\epsilon)$-tough 5-chordal graph of order $n$. If $$\delta_3(G)\ge n+2,$$ then $G$ is prism-Hamiltonian.
\end{corollary}

\begin{corollary}
Let $G$ be a $(1+\epsilon)$-tough 5-chordal graph such that, for every 3 mutual remote edges $e_1,e_2,e_3$ of $G$,
$$\sum^3_{i=0}d(e_i)>\frac{3}{2}(|V(G)|-3).$$
Then $G$ is prism-Hamiltonian.
\end{corollary}

\begin{corollary}
Let $G$ be a $(1+\epsilon)$-tough 5-chordal graph. If $d(e_1)+d(e_2)>|V(G)|-4$ for any remote edges $e_1,e_2$, then $G$ is prism-Hamiltonian.
\end{corollary}

\begin{remarks}
The results in this paper is almost trivial now. So, we are looking for some ``non-trivial'' condition for 5-chordal graphs to have edge-dominating cycles. Especially, we hope to find a toughness condition for 5-chordal graphs to have such cycles.
\end{remarks}

\bibliography{treftough}
\bibliographystyle{abbrv}
\end{document}